\documentclass[12pt,reqno]{amsart}

\usepackage{multirow}
\usepackage{hhline}

\usepackage{mathtools}
\usepackage{times}
\usepackage[T1]{fontenc}
\usepackage{mathrsfs}
\usepackage{latexsym}
\usepackage[dvips]{graphics}
\usepackage[titletoc, title]{appendix}
\usepackage{amsmath,amsfonts,amsthm,amssymb,amscd}
\usepackage[dvipsnames]{xcolor}
\usepackage{hyperref}
\usepackage{amsmath}
\usepackage[utf8]{inputenc}

\usepackage{color}
\usepackage{breakurl}

\usepackage{comment}
\newcommand{\bburl}[1]{\textcolor{blue}{\url{#1}}}

\newtheorem{thm}{Theorem}[section]

\newtheorem{cor}[thm]{Corollary}
\newtheorem{claim}[thm]{Claim}
\newtheorem{lem}[thm]{Lemma}
\newtheorem{prop}[thm]{Proposition}

\newtheorem{que}[thm]{Question}
\newtheorem{defi}[thm]{Definition}
\newtheorem{rek}[thm]{Remark}

\usepackage[utf8]{inputenc}

\DeclareFixedFont{\ttb}{T1}{txtt}{bx}{n}{12} 
\DeclareFixedFont{\ttm}{T1}{txtt}{m}{n}{12}  

\usepackage{color}
\definecolor{deepblue}{rgb}{0,0,0.5}
\definecolor{deepred}{rgb}{0.6,0,0}
\definecolor{deepgreen}{rgb}{0,0.5,0}

\usepackage{listings}

\newcommand\pythonstyle{\lstset{
language=Python,
basicstyle=\ttm,
morekeywords={self},              
keywordstyle=\ttb\color{deepblue},
emph={MyClass,__init__},          
emphstyle=\ttb\color{deepred},    
stringstyle=\color{deepgreen},
frame=tb,                         
showstringspaces=false
}}

\lstnewenvironment{python}[1][]
{
\pythonstyle
\lstset{#1}
}
{}

\newcommand\pythoninline[1]{{\pythonstyle\lstinline!#1!}}

\DeclareMathOperator{\supp}{supp}

\DeclareMathOperator{\sgn}{sgn}
\numberwithin{equation}{section}

\DeclareFontFamily{U}{mathx}{}
\DeclareFontShape{U}{mathx}{m}{n}{<-> mathx10}{}
\DeclareSymbolFont{mathx}{U}{mathx}{m}{n}
\DeclareMathAccent{\widehat}{0}{mathx}{"70}
\DeclareMathAccent{\widecheck}{0}{mathx}{"71}

\begin{document}

\title{Enlarge greedy sums in greedy-type properties by different factors}

\author[H. V. Chu]{H\`ung Vi\d{\^e}t Chu}
\email{\textcolor{blue}{\href{mailto:hchu@wlu.edu}{hchu@wlu.edu}}}
\address{Department of Mathematics, Washington and Lee University, Lexington, VA 24450, USA}

\subjclass[2020]{41A65, 46B15}

\keywords{bases, greedy algorithm, strong partially greedy, almost greedy}

\maketitle

\begin{abstract}
It was previously known that the almost greedy (AG) property essentially remains the same when we enlarge greedy sums in the classical definition by a factor $\lambda \geqslant 1$. The present paper shows that
if instead, we enlarge greedy sums in a reformulation of the AG property, we obtain a weaker one. However, the new property is essentially independent of the enlarging factor $\lambda$ once $\lambda > 1$. In contrast, we observe a continuum of partially greedy-like properties by varying $\lambda\in [1,\infty)$. Last but not least, under a threshold for $\lambda$, we characterize the isometric version of the weakened AG property. Specifically, the characterization holds if and only if $\lambda\in [1, 2]$. 
\end{abstract}

\tableofcontents

\section{Introduction}
Let $X$ be a separable Banach space over the field $\mathbb{F}\in \{\mathbb{R}, \mathbb{C}\}$ with the dual $X^*$. A sequence $(e_n)_{n=1}^\infty\subset X$ is called a basis of $X$ if the sequence is dense in $X$, and there exist $(e^*_n)_{n=1}^\infty\subset X^*$ such that $e^*_n(e_m) = 1$ if and only if $m = n$. A basis is semi-normalized if 
$$0 \ <\ \inf_n (\|e_n\|\wedge \|e_n^*\|)\ \leqslant\ \sup_n (\|e_n\|\vee \|e_n^*\|) \ <\ \infty.$$
We shall assume that our basis is semi-normalized, an essential condition to define the \textit{thresholding greedy algorithm} (TGA), which was introduced by Konyagin and Temlyakov \cite{KT1}. In particular, for $x\in X$ and $m\in \mathbb{N}$, the TGA produces the following $m$-term approximation of $x$:
$$G_m(x)\ :=\ \sum_{n\in A} e_n^*(x)e_n,$$
where $A$ satisfies
$$|A| \ =\ m \mbox{ and }\min_{n\in A}|e_n^*(x)|\ \geqslant\ \max_{n\notin A}|e_n^*(x)|.$$
The set $A$ is called a greedy set of $x$ of order $m$, which may not necessarily be unique. Let $\mathcal{G}(x,m)$ denote the collection of all greedy sets of $x$ of order $m$. Note that the TGA is well-defined because $\lim_{n\rightarrow\infty} |e_n^*(x)| = 0$ for every $x$ thanks to semi-normalization. 

Denote the collection of all finite subsets of $\mathbb{N}$ by $\mathbb{N}^{<\infty}$. Konyagin and Temlyakov \cite{KT1} studied the approximation efficiency of the TGA by comparing the error $\|x-G_m(x)\|$ against all possible $m$-term approximations. They called a basis \textit{greedy} if there is a constant $C \geqslant 1$ such that
$$\|x-G_m(x)\|\ \leqslant\ C\inf_{\substack{A\in \mathbb{N}^{<\infty}, |A| = m\\ (a_n)_{n\in A}\subset \mathbb{F}}}\left\|x-\sum_{n\in A}a_n e_n\right\|, \forall x\in X, m\in \mathbb{N},  G_m(x).$$
Here is a beautiful theorem that characterizes greedy bases; the readers can see Definitions \ref{uncondefi} and \ref{democraticdefi} for unconditional and democratic bases, respectively. 

\begin{thm}[Konyagin and Temlyakov \cite{KT1}]\label{chargreedy}
A basis is greedy if and only if it is unconditional and democratic. 
\end{thm}

It turns out that greedy bases are quite rare as many classical Banach spaces including $\ell_p \oplus \ell_q (p\neq q)$, $(\oplus_{n=1}^\infty \ell_p^n)_{\ell_1} (1 < p \leqslant \infty)$, and $(\oplus_{n=1}^\infty \ell_p^n)_{c_0} (1 \leqslant p < \infty)$  have no greedy bases (see \cite{DFOS, EW} and the discussion on \cite[page 5]{C4}). This rarity of greedy bases motivates the definition of almost greedy (AG) bases by Dilworth et al.\ \cite{DKKT}. A basis is said to be \textit{almost greedy} if there is a constant $C \geqslant 1$ such that
$$\|x-G_m(x)\|\ \leqslant\ C\inf_{A\in \mathbb{N}^{<\infty}, |A| = m}\|x-P_A(x)\|, \forall x\in X, m\in \mathbb{N}, G_m(x),$$
where $P_A(x) = \sum_{n\in A} e_n^*(x)e_n$, the projection of $x$ onto $A$. The abundance of AG bases is illustrated by several results in \cite{DKK}: for example, \cite[Remark 7.7]{DKK} states that if a Banach space $X$ contains a complemented copy of $\ell_p$ for some $1 \leqslant p < \infty$, then $X$ has infinitely many inequivalent AG bases. In the same spirit as Theorem \ref{chargreedy}, we have 
\begin{thm}[Dilworth et al.\ \cite{DKKT}]\label{thmDKKT}
A basis is AG if and only if it is quasi-greedy and democratic. 
\end{thm}
We refer the readers to Definition \ref{quasidefi} for the definition of quasi-greedy bases.  
Greedy-type properties have been investigated from different angles, including the (non-)existence of greedy/AG bases \cite{AABW, AADK, DFOS, DKK}, the isometric theory \cite{AA0, AA, AABCO, AW}, and Lebesgue inequalities \cite{BBL, BBG, DKO}. 

The present paper continues \cite{C1, C3, C5, DKKT}, which aim to study the effect of changing the size of greedy sums relative to the size of the optimal approximation. We recall a surprising result by Dilworth et al. \cite{DKKT}
that enlarging the greedy sums in the greedy condition by a factor $\lambda > 1$ gives us a property equivalent to the AG property; that is, a basis is AG if and only if there is a constant $C \geqslant 1$ such that
\begin{equation}\label{e2}\|x-G_{\lceil \lambda m\rceil}(x)\|\ \leqslant\ C\inf_{\substack{A\in \mathbb{N}^{<\infty}, |A| = m\\ (a_n)_{n\in A}\subset \mathbb{F}}}\left\|x-\sum_{n\in A}a_n e_n\right\|, \forall x\in X, m\in \mathbb{N}, G_{\lceil \lambda m\rceil}(x).\end{equation}
Hence, enlarging greedy sums in the greedy property by the factor $\lambda$ weakens the property quite dramatically.   
Motivated by this result, \cite{C1} studied whether enlarging greedy sums weakens the AG property and the so-called strong partially greedy (PG) property. Here, the PG property is defined as the existence of a constant $C \geqslant 1$ such that
$$\|x-G_m(x)\|\ \leqslant\ C\inf_{0\leqslant n\leqslant m}\left\|x-\sum_{i=1}^n e_i^*(x)e_i\right\|, \forall x\in X,  m\in \mathbb{N},  G_m(x).$$
This property was first introduced by Dilworth et al.\ \cite{DKKT} for Schauder bases and then extended to general bases by Berasategui et al.\ \cite{BBL}; the idea is to compare the greedy sums against partial sums. 

For ease of stating the main findings of \cite{C1}, we give a name to the properties arising from enlarging greedy sums. 
\begin{defi}\cite[Definition 1.3]{C1} \normalfont
Let $\lambda\geqslant 1$. A basis $(e_n)_{n=1}^\infty$ is said to be $\lambda$-AG if there exists a constant $C \geqslant 1$ such that
$$\|x-G_{\lceil \lambda m\rceil}(x)\|\ \leqslant\ C\inf_{A\in \mathbb{N}^{<\infty}, |A| = m}\|x-P_A(x)\|, \forall x\in X, m\in \mathbb{N},  G_{\lceil \lambda m\rceil}(x).$$
\end{defi}

\begin{defi}\cite[Definition 1.5]{C1}\normalfont\label{defLambdaPG}
Let $\lambda\geqslant 1$. A basis $(e_n)_{n=1}^\infty$ is said to be $\lambda$-partially greedy ($\lambda$-PG) if there exists a constant $C\geqslant 1$ such that 
$$\|x-G_{\lceil \lambda m\rceil}(x)\|\ \leqslant\ C\inf_{0\leqslant n\leqslant m}\left\|x-\sum_{i=1}^n e_i^*(x)e_i\right\|, \forall x\in X, m\in \mathbb{N},  G_{\lceil \lambda m\rceil}(x).$$
\end{defi}

\begin{thm}\cite[Theorem 1.4]{C1}\label{pm2}
A basis is AG if and only if it is $\lambda$-AG for all (some) $\lambda\geqslant 1$.
\end{thm}

\begin{thm}\cite[Theorem 1.8]{C1}\label{pm3}
Let $\lambda > 1$. 
There exists an unconditional basis that is $\lambda$-PG but is not PG. 
\end{thm}

While Theorem \ref{pm2} says that enlarging greedy sums does not weaken the AG property, Theorem \ref{pm3} states that enlarging greedy sums does weaken the PG property. The present paper adds to the literature on enlarging greedy sums by addressing the two following questions. Throughout the paper, let $\lambda_1, \lambda_2$ denote two real numbers such that $1\leqslant \lambda_1 < \lambda_2$. Let $\lambda$ be any real number at least $1$. 

\begin{que}\normalfont\label{pr1}
It follows from previous work that a $\lambda_1$-PG basis must be $\lambda_2$-PG. Is it true that for an arbitrary pair $(\lambda_1, \lambda_2)$, a $\lambda_2$-PG basis must be $\lambda_1$-PG? Our Theorem \ref{m1} answers this question in the negative. 
\end{que}

\begin{rek}\normalfont
In an attempt to give a sufficient condition on when a $\lambda_2$-PG basis must be $\lambda_1$-PG, \cite[Definition 8.9]{C4} introduced right-skewed bases. Since the basis in the proof of \cite[Theorem 1.8]{C1} is right-skewed, it is $\lambda$-PG basis for every $\lambda > 1$. The basis we shall construct to answer Question \ref{pr1} must not be right-skewed.
\end{rek}

\begin{que}\normalfont\label{pr2}
Out of all the greedy-type properties mentioned above, the AG property is the only one not weakened by enlarging greedy sums (Theorem \ref{pm2}). Is there an equivalent reformulation of the AG property such that when we enlarge greedy sums in the reformulation, we obtain a weaker property? Theorem \ref{m10} answers this question in the positive. 
\end{que}

On the way to answer Questions \ref{pr1} and \ref{pr2}, we encounter various results that are of their own interest. For example, from \cite{BBC, C3}, we have the $\lambda$-AG2 and $\lambda$-PG2 properties. While
$$1\mbox{-PG} \ =\  1\mbox{-PG2},$$
we establish that for all $\lambda_2 > \lambda_1 > 1$ and $\lambda'_1, \lambda'_2, \lambda''_1, \lambda''_2 > 1$, we have
\begin{equation}\label{e25}\lambda''_1\mbox{-AG1}\ \Longleftrightarrow\ \lambda''_2\mbox{-AG2}\ \Longrightarrow\ \lambda'_1\mbox{-PG2}\ \Longleftrightarrow\ \lambda'_2\mbox{-PG2}\ \Longrightarrow\ \lambda_1\mbox{-PG}\ \Longrightarrow\ \lambda_2\mbox{-PG};\end{equation}
furthermore, none of the one-sided arrows can be reversed. Another example is the isometric result, Theorem \ref{m11}.

\section{Main results and preliminaries}\label{mainprelim}

In this section, we record our main results along with useful preliminaries. We start with the first result that answers Question \ref{pr1}. 

\begin{thm} \label{m1} Let $1\leqslant \lambda_1 < \lambda_2$. 
There exists an unconditional basis that is $\lambda_2$-PG but not $\lambda_1$-PG, i.e., we have a continuum of PG-type properties.  
\end{thm}

We remark that Theorem \ref{m1} differentiates the case of PG bases from the case of greedy bases, where we obtain the same property regardless of the $\lambda$ in \eqref{e2} as long as $\lambda > 1$.

We now turn to Question \ref{pr2} by first recalling relevant previous work. Companions to the PG bases are the so-called reverse partially greedy (RPG) bases introduced by Dilworth and Khurana \cite{DK}. Like the AG property, enlarging greedy sums in the original definition of RPG bases does not weaken the property. To obtain an analog of Theorem \ref{pm3} for RPG bases, \cite[Theorem 1.1]{C3} ``intervalizes" the property. We, therefore, suspect that the reformulation asked by Question \ref{pr2} must be an ``interval version" of the AG property. A natural candidate is the reformulation in \cite{BBC}. Denote the collection of all finite intervals of $\mathbb{N}$ by $\mathcal{I}$, which includes the empty set. 

\begin{thm}\cite[Theorem 2.8]{BBC}\label{BBCtheo}
A basis is almost greedy if and only if there is $C\geqslant 1$ such that
\begin{equation}\label{e40}\|x-G_m(x)\|\ \leqslant\ C\inf_{I\in \mathcal{I},|I|\leqslant m}\|x-P_I(x)\|,\forall x\in X,  m\in \mathbb{N},  G_m(x).\end{equation}
\end{thm}

For conciseness, let us give a name to the Condition \eqref{e40} with enlarged greedy sums. 

\begin{defi}\normalfont\label{defiAG2}
A basis is said to be $\lambda$-almost greedy of type $2$ ($\lambda$-AG2, for short) if there is $C\geqslant 1$ such that
\begin{equation}\label{e5}\|x-G_{\lceil \lambda m\rceil}(x)\|\ \leqslant\ C\inf_{I\in \mathcal{I},|I|\leqslant m}\|x-P_I(x)\|,\forall x\in X, m\in \mathbb{N},  G_{\lceil \lambda m\rceil}(x).\end{equation}
\end{defi}

We can restate Theorem \ref{BBCtheo} as follows: a basis is AG if and only if it is $1$-AG2.
We answer Question \ref{pr2}. 
\begin{thm}\label{m10}
Let $\lambda > 1$. An AG basis is $\lambda$-AG2; however, there exists a $\lambda$-AG2 basis that is not AG. 
\end{thm}

With Theorem \ref{m10}, we now know that every greedy-type property: greedy, AG, PG, and PRG is weakened when the greedy sums are enlarged in either its original definition or its reformulation.  
Unlike Theorem \ref{m1}, we do not have a continuum of AG-type properties, as illustrated by the following proposition.
\begin{prop}\label{m2}
Let $1 < \lambda_1 < \lambda_2$. A basis is $\lambda_1$-AG2 if and only if it is $\lambda_2$-AG2. 
\end{prop}

As an intermediate step to prove Theorem \ref{m10} and Proposition \ref{m2}, we characterize $\lambda$-AG2 bases, using what we called the $\lambda$-symmetry for largest coefficients of type $2$ ($\lambda$-SLC2) (Definition \ref{SLC2defi}) and $\lambda$-democracy of type $2$ (Definition \ref{lambdademo}).

\begin{thm}\label{m2'}
Let $\lambda\geqslant 1$. The following are equivalent:
\begin{enumerate}
\item [i)] a basis is $\lambda$-AG2;
\item [ii)] a basis is quasi-greedy and $\lambda$-SLC2;
\item [iii)] a basis is quasi-greedy and $\lambda$-democratic of type $2$. 
\end{enumerate}
\end{thm}

Here is a corollary of Proposition \ref{m2} and Theorem \ref{m2'}:
\begin{cor}\label{ocor}
The following are equivalent:
\begin{enumerate}
\item [i)] a basis is $\lambda$-AG2 for (some) all $\lambda > 1$;
\item [ii)] a basis is quasi-greedy and $\lambda$-SLC2 for (some) all $\lambda > 1$;
\item [iii)] a basis is quasi-greedy and $\lambda$-democratic of type $2$ for (some) all $\lambda > 1$. 
\end{enumerate}
\end{cor}

We are also interested in the isometric theory. We show that the equivalence between ``$\lambda$-AG2 with constant $1$" and ``$\lambda$-SLC2 with constant $1$" holds if and only if $\lambda\leqslant 2$. Formally, 

\begin{thm}\label{m11}
For every $\lambda \geqslant 1$, we have the implication 
$$\lambda\mbox{-AG2 with constant 1}\Longleftarrow\ \lambda\mbox{-SLC2 with constant 1}.$$

For $\lambda \leqslant 2$, the following equivalence holds
$$\lambda\mbox{-AG2 with constant 1}\Longleftrightarrow\ \lambda\mbox{-SLC2 with constant 1}.$$

For every $\lambda > 2$, there exists a basis that is $\lambda$-AG2 with constant $1$ but is not $\lambda$-SLC2 with constant $1$. 
\end{thm}

Our paper is structured as follows: Section \ref{differentiatePG} proves Theorem \ref{m1}; Section \ref{AG2} first proves Theorem \ref{m2'}, which will then be used to establish Theorem \ref{m10} and Proposition \ref{m2}; Section \ref{isometry} is devoted to Theorem \ref{m11}; finally, Section \ref{(R)PG} addresses PG and RPG bases, including the implications in \eqref{e25}.

We conclude this section by recalling several building blocks of the greedy-type properties mentioned earlier.

\begin{defi}\label{uncondefi}\normalfont
A basis $(e_n)_{n=1}^\infty$ is said to be unconditional with constant $C$ if we have
$\|\sum_{n=1}^N a_n e_n\| \leqslant C\|\sum_{n=1}^N b_ne_n\|$,
for all $N\in \mathbb{N}$ and for all scalar sequences $(a_n)_{n=1}^N$ and $(b_n)_{n=1}^N$ with $|a_n|\leqslant |b_n|$.
\end{defi}

For $A\in \mathbb{N}^{<\infty}$, we write $1_A$ to mean the vector $\sum_{n\in A}e_n$. 

\begin{defi}\label{democraticdefi}\normalfont
A basis $(e_n)_{n=1}^\infty$ is said to be democratic with constant $C$ if $\|1_A\| \leqslant C\|1_B\|$, 
for all $A, B\in \mathbb{N}^{<\infty}$ with $|A|\leqslant |B|$.
\end{defi}

\begin{defi}\label{quasidefi}\normalfont
A basis is said to be quasi-greedy with constant $C$ if 
$\|G_m(x)\|\leqslant C\|x\|$ for all $x\in X, m\in \mathbb{N}$, and $G_m(x)$. 

A basis is said to be suppression quasi-greedy with constant $C$ if 
$\|x-G_m(x)\|\leqslant C\|x\|$ for all $x\in X, m\in \mathbb{N}$, and $G_m(x)$. 
\end{defi}

Clearly, an unconditional basis is quasi-greedy; however, there exists a quasi-greedy basis that is not unconditional \cite{KT1}.

\section{The space $X_{\lambda_1, \lambda_2}$ whose basis is $\lambda_2$-PG but not $\lambda_1$-PG}\label{differentiatePG}

In this section, we prove Theorem \ref{m1}. To do so, we construct a space, denoted by $X_{\lambda_1, \lambda_2}$, whose unit vector basis is unconditional with constant $1$ and is $\lambda_2$-PG but not $\lambda_1$-PG. We need the following characterization of $\lambda$-PG, which involves the $\lambda$-max conservative property. We write $A < B$ to mean that $a < b$ for all $a\in A$ and $b\in B$ and write $b < A$ to mean that $b < a$ whenever $a\in A$. Similar notation applies for other inequalities.  

\begin{defi}\normalfont
A basis is $\lambda$-max conservative if there exists $C > 0$ such that 
$$\|1_A\|\ \leqslant\ C\|1_B\|, \forall A, B\in \mathbb{N}^{<\infty}, A < B, (\lambda-1)\max A + |A|\leqslant |B|.$$
\end{defi}

\begin{thm}\cite[Theorem 1.7]{C1}\label{pm1} Let $\lambda\geqslant 1$. A basis is $\lambda$-PG if and only if it is quasi-greedy and $\lambda$-max conservative. 
\end{thm}
In light of Theorem \ref{pm1}, it suffices to construct an  unconditional basis that is $\lambda_2$-max conservative but is not $\lambda_1$-max conservative. 

Since $1\leqslant \lambda_1 < \lambda_2$, it is possible to choose $a_1, a_2, a_3$, and $a_4$ that satisfy
$$a_1\ >\ 1, \quad 1 \ >\ a_3, a_4 \ >\ 0,  \mbox{ and }$$
$$ a_3 + (\lambda_1-1)(a_1 + a_3)\ <\ a_2 \ <\ \lambda_2 - 1 - a_4.$$
Let $p$ be such that
$$p \ >\ \max\left\{\frac{a_2}{a_1-1}, \frac{a_2+1}{a_4}, a_1+a_3\right\}.$$
Define the sequence $(g_n)_{n=1}^\infty$ recursively: let  $g_1 = 1$. Supposing that $g_n$ has been defined for some $n\geqslant 1$, we choose $g_{n+1}$ so that
\begin{equation}\label{e4}g_{n+1} > (p+n)g_n.\end{equation}

We now define the space $X_{\lambda_1,\lambda_2}$: let
$$\mathcal{F} \ =\ \{F\in \mathbb{N}^{<\infty}\,:\, F\neq \emptyset\mbox{ and }\min F = |F| = g_i\mbox{ for some }i\}, \mbox{ and}$$
define $X_{\lambda_1, \lambda_2}$ to be the completion of $c_{00}$ with respect to the norm
$$\|(x_i)_{i=1}^\infty\|\ =\ \max_i |x_i|\vee \sup_{\substack{j\in \mathbb{N}, F\in \mathcal{F}\\ \min F = g_j}}\sum_{\substack{i\in F\\ i > a_1g_j\\ i\notin [g_{j+1}, g_{j+1}+a_2g_j]}}|x_i|.$$
Let $\mathcal{B}_{\lambda_1, \lambda_2} = (e_n)_{n=1}^\infty$ be the standard unit vector basis of $X_{\lambda_1, \lambda_2}$. Clearly, $\mathcal{B}_{\lambda_1, \lambda_2} $ is normalized and unconditional with constant $1$. 

\begin{thm}\label{m1'}
The unit vector basis $\mathcal{B}_{\lambda_1, \lambda_2} $ of the space $X_{\lambda_1, \lambda_2}$ is unconditional and $\lambda_2$-max conservative, thus $\lambda_2$-PG. However, $\mathcal{B}_{\lambda_1, \lambda_2} $ is not $\lambda_1$-max conservative and thus, not $\lambda_1$-PG.
\end{thm}

We prove Theorem \ref{m1'} through Claims \ref{cl1} and \ref{cl2}.

\begin{claim}\label{cl1}
The basis $\mathcal{B}_{\lambda_1, \lambda_2} $ is $\lambda_2$-max conservative. 
\end{claim}

\begin{proof}
Pick $A, B\in \mathbb{N}^{<\infty}$ with $A < B$ and $(\lambda_2-1)\max A + |A| \leqslant |B|$. Choose $i_0$ such that 
$g_{i_0}\leqslant \max A < g_{i_0+1}$. It follows trivially from the definition of the norm that 
\begin{equation}\label{e1}
\|1_A\|\ \leqslant\ \max\{1, g_{i_0}-1\}.
\end{equation}
If $g_{i_0} < 2a_4^{-1}-1$, then 
$$\|1_A\|\ \leqslant\ \max\{1, 2a_4^{-1}-2\}\ \leqslant\ \max\{1, 2a_4^{-1}-2\}\|1_B\|.$$
If $i_0 \leqslant 2$, then \eqref{e1} trivially gives
$$\|1_A\|\ \leqslant\ g_2-1\ \leqslant\ (g_2-1) \|1_B\|.$$

Assume that $g_{i_0}\geqslant 2a_4^{-1}-1$ and $i_0\geqslant 3$. Set $B':= B\cap (\max A, g_{i_0+1})$.

Case 1: $|B'| \geqslant a_4\max A\geqslant a_4g_{i_0}$. By \eqref{e4}, we have
\begin{align*}
|\{n\in B'\,:\, n  > g_{i_0} + a_2 g_{i_0-1}\}|& \ =\ |B'| - |\{n\in B'\,:\, n\leqslant g_{i_0} + a_2 g_{i_0-1}\}|\\
&\ \geqslant\ a_4 g_{i_0} - a_2 g_{i_0-1}\\
&\ >\ (a_2+1)g_{i_0-1}-a_2g_{i_0-1}\ =\ g_{i_0-1}.
\end{align*}
Hence, $\|1_{B'}\|\geqslant g_{i_0-1}-1$. 
\begin{itemize}
\item If $\max A\leqslant a_1g_{i_0}$, then 
$$\|1_A\|\ \leqslant\  g_{i_0-1}-1\ \leqslant\ \|1_{B'}\|.$$ 
\item If $\max A > a_1 g_{i_0}$, then 
$$\|1_A\|\ \leqslant\ g_{i_0}-1\ \leqslant\ \frac{2}{a_4}(a_4 g_{i_0}-1)\ \leqslant\ \frac{2}{a_4}\|1_{B'}\|,$$
because $|B'|\geqslant a_4 g_{i_0}$, $a_4 g_{i_0} < g_{i_0}$, and $\min B > a_1g_{i_0}$. 
\end{itemize}

Case 2: $|B'| < a_4\max A$. Then $B'' := B\cap [g_{i_0+1}, \infty)$ has 
$$|B''| \ >\ (\lambda_2-1- a_4)\max A + |A|.$$
It follows that
\begin{align*}
|\{n\in B\,:\, n > g_{i_0+1} + a_2g_{i_0}\}|&\ =\ |B''| - |\{n\in B''\,:\, n \leqslant g_{i_0+1} + a_2g_{i_0}\}|\\
&\ >\ (\lambda_2-1- a_4)\max A + |A| - a_2g_{i_0}-1\\
&\ \geqslant\ (\lambda_2 - 1 - a_2 - a_4)g_{i_0}.
\end{align*}
Hence, 
$$\|1_B\|\ \geqslant\ (g_{i_0}-1)\wedge \lceil (\lambda_2-1-a_2-a_4)g_{i_0}\rceil\ \geqslant\ \left(1 \wedge (\lambda_2-1-a_2-a_4)\right)\|1_A\|.$$
This completes our proof. 
\end{proof}

\begin{claim}\label{cl2}
The basis $\mathcal{B}_{\lambda_1, \lambda_2} $ is not $\lambda_1$-max conservative.
\end{claim}
\begin{proof}
Pick $j\in \mathbb{N}$. Choose 
\begin{align*}
A &\ =\ \{\lfloor a_1 g_j\rfloor +1, \lfloor a_1 g_j\rfloor + 2\ldots, \lfloor a_1 g_j\rfloor + \lfloor a_3 g_j\rfloor\}\mbox{ and }\\
B&\ =\ \{g_{j+1}, g_{j+1}+1,  \ldots, g_{j+1}+ \lfloor a_3 g_j\rfloor+ \lfloor (\lambda_1-1)(\lfloor a_1 g_j\rfloor + \lfloor a_3 g_j\rfloor)\rfloor\}.
\end{align*}
Then $B > A$ because
$$\max A\ = \ \lfloor a_1 g_j\rfloor + \lfloor a_3 g_j\rfloor\ \leqslant\ (a_1 + a_3)g_j \ <\ g_{j+1}.$$
Furthermore, 
\begin{align*}
|B| &\ =\ \lfloor a_3 g_j\rfloor+ \lfloor (\lambda_1-1)(\lfloor a_1 g_j\rfloor + \lfloor a_3 g_j\rfloor)\rfloor + 1\\
&\ \geqslant\ |A| + (\lambda_1-1)(\lfloor  a_1 g_j\rfloor + \lfloor a_3 g_j\rfloor)\\
&\ =\ |A| + (\lambda_1-1) \max A.
\end{align*}
We have $\|1_A\|\geqslant \lfloor a_3 g_j\rfloor$ because $|A| = \lfloor a_3 g_j\rfloor$ and $\min A > a_1g_j$. Meanwhile, it follows from
$$\lfloor a_3 g_j\rfloor+ \lfloor (\lambda_1-1)(\lfloor a_1 g_j\rfloor + \lfloor a_3 g_j\rfloor)\rfloor \ \leqslant\ a_3 g_j + (\lambda_1-1)(a_1 g_j + a_3 g_j) \ <\ a_2 g_j$$ and
$$g_{j+1}+\lfloor a_3 g_j\rfloor+ \lfloor (\lambda_1-1)(\lfloor a_1 g_j\rfloor + \lfloor a_3 g_j\rfloor)\rfloor\ <\ g_{j+1}+a_2g_{j}\ <\ a_1g_{j+1}$$
that $\|1_B\|\leqslant g_{j-1}$. Hence, $\|1_A\|/\|1_B\|\rightarrow\infty$ as $j\rightarrow \infty$. Therefore, $\mathcal{B}_{\lambda_1, \lambda_2} $ is not $\lambda_1$-max conservative. 
\end{proof}

\section{On $\lambda$-almost greedy of type $2$ bases}\label{AG2}

Our first goal in this section is to break the $\lambda$-AG2 property into simpler properties, namely quasi-greedy and $\lambda$-SLC2 (Theorem \ref{m2'}). The characterization then facilitates our proof of Theorem \ref{m10} and Proposition \ref{m2}.

We shall need the following notation. 
A sign sequence $\varepsilon = (\varepsilon_n)_{n=1}^\infty$ is a sequence of scalars of modulus $1$, i.e., $|\varepsilon_n| = 1$. For $A\in \mathbb{N}^{<\infty}$, let
\begin{itemize}
\item $1_{\varepsilon A}:= \sum_{n\in A} \varepsilon _n e_n$,
\item $\|x\|_\infty := \max_{n} |e_n^*(x)|$,
\item $s(A) = \begin{cases}0, &\mbox{ if } A = \emptyset,\\
\max A - \min A + 1, &\mbox{ if } A\neq \emptyset,\end{cases}$ and 
\item $\supp(x) = \{n: e^*_n(x)\neq 0\}$.
\end{itemize}
For $A, B\in \mathbb{N}^{<\infty}$, we say, ``$B$ surrounds $A$", if either $A = \emptyset$ or $B\cap [\min A, \max A] = \emptyset$. We say, ``$x$ surrounds $A$", if $\supp(x)$ surrounds $A$.

\subsection{Characterization of $\lambda$-AG2 bases}
\begin{defi}\label{SLC2defi}\normalfont
For $\lambda\geqslant 1$, a basis is said to be $\lambda$-symmetric for largest coefficients of type $2$ ($\lambda$-SLC2) if there exists $C \geqslant 1$ such that 
\begin{equation}\label{e8}\|x+1_{\varepsilon A}\|\leqslant C\|x+1_{\delta B}\|,\end{equation}
for all $x\in X$ with $\|x\|_\infty\leqslant 1$, for all signs $\varepsilon, \delta$, and for all $A, B\in \mathbb{N}^{<\infty}$ such that
\begin{itemize}
\item[a)] $(\lambda-1)s(A) +|A|\leqslant |B|$,
\item[b)] $B\cap \supp(x) = \emptyset$, and
\item[c)] $B$ and $x$ surround $A$.
\end{itemize}
\end{defi}

\begin{prop}\label{SLC2redefi}
A basis is $\lambda$-SLC2 with constant $C$ if and only if 
\begin{equation}\label{e7}
\|x\| \ \leqslant\ C\|x-P_A(x)+1_{\varepsilon B}\|,
\end{equation}
for all $x\in X$ with $\|x\|_\infty\leqslant 1$, for all signs $\varepsilon$, and for all $A, B\in \mathbb{N}^{<\infty}$ such that
\begin{itemize}
\item[a)] $(\lambda-1)s(A) +|A|\leqslant |B|$,
\item[b)] $B\cap \supp(x-P_A(x)) = \emptyset$, and
\item[c)] $B$ and $x-P_A(x)$ surround $A$.
\end{itemize}
\end{prop}

\begin{proof}
Assume that our basis is $\lambda$-SLC2 with constant $C$. Let $x, A, B, \varepsilon$ be chosen as in Proposition \ref{SLC2redefi}. By convexity and \eqref{e8},
$$
\|x\|\ \leqslant\ \sup_{\delta}\|x - P_A(x) +1_{\delta A}\|\ \leqslant\ C\|x-P_A(x) + 1_{\varepsilon B}\|
$$
Conversely, assume that our basis satisfies \eqref{e7}. Let $x, A, B, \varepsilon, \delta$ be chosen as in Definition \ref{SLC2defi}. Let $y = x+1_{\varepsilon A}$. By \eqref{e7}, 
$$
\|x+1_{\varepsilon A}\|\ =\ \|y\|\ \leqslant\ C\|y-P_A(y)+1_{\delta B}\|\ =\ C\|x+1_{\delta B}\|.
$$
We have completed the proof. 
\end{proof}

\begin{lem}\label{l1}
If a basis is $\lambda$-AG2, then it is quasi-greedy and $\lambda$-SLC2.
\end{lem}

\begin{proof}
Assume that our basis is $\lambda$-AG2 with constant $C_1$. Then by setting $I = \emptyset$ in \eqref{e5}, we obtain 
$$\|x-G_{\lceil\lambda m\rceil}(x)\|\ \leqslant\ C_1\|x\|, \forall x\in X, m \in\mathbb{N}, \forall G_{\lceil\lambda m\rceil}(x).$$
By \cite[Proposition 4.1]{O}, we know that the basis is quasi-greedy, i.e., there exists $C_2$ such that 
\begin{equation}\label{e6}\|x-G_m(x)\|\ \leqslant\ C_2\|x\|,\forall x\in X, \forall m\in \mathbb{N}, \forall G_m(x).\end{equation}

Let $x, A, B, \varepsilon, \delta$ be chosen as in Definition \ref{SLC2defi}. We shall show that 
$$\|x+1_{\varepsilon A}\|\ \leqslant\ C_1C_2\|x+1_{\delta B}\|.$$

If $A = \emptyset$, then \eqref{e6} gives
$$\|x+1_{\varepsilon A}\|\ =\ \|x\|\ \leqslant\ C_2\|x+1_{\delta B}\|.$$

If $A\neq \emptyset$, then we set $y:= x + 1_{\varepsilon A}+1_{\delta B} + 1_D$, where 
$D = [\min A, \max A]\backslash A$. Observe that $B\cup D$ is a greedy set of $y$, and 
$$|B\cup D| \ =\ |B| + |D|\ \geqslant\ (\lambda-1)s(A) + |A| + (s(A)-|A|)\ =\ \lambda s(A).$$
Choose $\Lambda\subset B\cup D$ such that $|\Lambda| = \lceil \lambda s(A)\rceil$. Since the basis is $\lambda$-AG2 with constant $C_1$, we have
\begin{align*}
\|x+1_{\varepsilon A}\|&\ =\ \|y-P_{B\cup D}(y)\|\\
&\ \leqslant\ C_2\|y - P_{\Lambda}(y)\|\\
&\ \leqslant\ C_1C_2\|y- P_{A\cup D}(y)\|\ =\ C_1C_2\|x+1_{\delta B}\|, 
\end{align*}
as desired. 
\end{proof}

\begin{lem}\label{l2}
If a basis is suppression quasi-greedy with constant $C_1$ and $\lambda$-SLC2 with constant $C_2$, then it is $\lambda$-AG2 with constant $C_1C_2$. 
\end{lem}

Before proving Lemma \ref{l2}, we recall the uniform boundedness of the truncation operator: for each $a > 0$, we define the truncation function $T_a$ as follows: for $b\in \mathbb{F}$, 
$$T_a(b) \ =\ \begin{cases}\sgn(b) a, &\mbox{ if }|b| > a,\\ b, &\mbox{ if }|b| \leqslant a\end{cases},$$
where $\sgn(b) = b/|b|$ for $b\neq 0$. 
With an abuse of notation, we define the truncation operator $T_a: X\rightarrow X$ as
$$T_a(x)\ =\ \sum_{n=1}^\infty T_a(e_n^*(x))e_n\ =\ a 1_{\varepsilon \Lambda_a(x)} + P_{\mathbb{N}\backslash  \Lambda_a(x)}(x),$$
where $\Lambda_a(x) = \{n: |e_n^*(x)| > a\}$ and $\varepsilon = (\sgn(e_n^*(x)))_n$.

\begin{lem}\cite[Lemma 2.5]{BBG}\label{BBGlem} If a basis is suppression quasi-greedy with constant $C$, 
then $\|T_a\|\leqslant C$ for every $a > 0$.
\end{lem}

\begin{proof}[Proof of Lemma \ref{l2}]
Assume that our basis is suppression quasi-greedy with constant $C_1$ and $\lambda$-SLC2 with constant $C_2$. Let $x\in X$, $m\in \mathbb{N}$, and $\Lambda\in G(x, \lceil \lambda m\rceil)$. Choose $I\in \mathcal{I}$ and $|I|\leqslant m$. 
If $I = \emptyset$, then $\|x-P_{\Lambda}(x)\|\leqslant C_1\|x\| = C_1\|x-P_I(x)\|$. Assume that $I\neq \emptyset$. 
We shall show that
$$\|x-P_{\Lambda}(x)\|\ \leqslant\ C_1C_2\|x-P_I(x)\|.$$
Let $a = \min_{n\in \Lambda}|e_n^*(x)|$. Then $\|x-P_\Lambda(x)\|_\infty\leqslant a$. By the standard perturbation argument, we can assume that $a > 0$. 

We now verify that $x-P_\Lambda(x)$, $\Lambda\backslash I$, and $I\backslash \Lambda$ satisfy the conditions in Proposition \ref{SLC2redefi}. First, we have
\begin{align*}
|\Lambda\backslash I|\ =\ |\Lambda|-|\Lambda\cap I|&\ \geqslant\ \lambda m + (|I|-|\Lambda\cap I|)-|I|\\
&\ \geqslant\ \lambda m + |I\backslash \Lambda| - m\ \geqslant\ (\lambda-1)s(I\backslash \Lambda) + |I\backslash \Lambda|.
\end{align*}
Secondly, $(\Lambda\backslash I)\cap \supp(x-P_{\Lambda}(x)-P_{I\backslash \Lambda}(x)) = \emptyset$.
Lastly, $\Lambda\backslash I$ surrounds $I\backslash \Lambda$ because $I\in \mathcal{I}$. Obviously, $x-P_{\Lambda}(x) - P_{I\backslash \Lambda}(x) = x - P_{\Lambda\cup I}(x)$ surrounds $I\backslash \Lambda$. Let $\varepsilon = (\varepsilon_n)_n$, where $\varepsilon_n = \begin{cases}0, &\mbox{ if }e_n^*(x) = 0,\\ \sgn(e_n^*(x)), &\mbox{ if }e_n^*(x) \neq 0\end{cases}$. By \eqref{e7} and Lemma \ref{BBGlem}, we obtain
\begin{align*}
\|x-P_{\Lambda}(x)\|&\ \leqslant\ C_2\|x-P_{\Lambda}(x)-P_{I\backslash \Lambda}(x) + a1_{\varepsilon \Lambda\backslash I}\|\\
&\ =\ C_2\|T_{a}(x-P_{\Lambda}(x)-P_{I\backslash \Lambda}(x) + P_{\Lambda\backslash I}(x))\|\\
&\ \leqslant\ C_1C_2\|x-P_{I}(x)\|.
\end{align*}
This completes our proof. 
\end{proof}

We can simplify our characterization by introducing $\lambda$-democratic of type $2$ bases. 
\begin{defi}\label{lambdademo}\normalfont
A basis is said to be $\lambda$-democratic of type $2$ if there is a constant $C > 0$ such that
$\|1_A\|\leqslant C\|1_B\|$ whenever $A, B\in \mathbb{N}^\infty$ with $(\lambda-1)s(A) + |A| \leqslant |B|$ and $B$ surrounding $A$. 
\end{defi}

\begin{prop}\label{p1}
A quasi-greedy and $\lambda$-democratic of type $2$ basis is $\lambda$-SLC2. 
\end{prop}

To prove Proposition \ref{p1}, we need the UL-property (\cite{DKKT}) of quasi-greedy bases: if a basis is quasi-greedy with constant $C$, then 
\begin{equation}\label{e9}\frac{1}{2C}\min |a_n|\|1_A\|\ \leqslant\ \left\|\sum_{n\in A}a_n e_n\right\|\ \leqslant\ 2C\max|a_n|\|1_A\|,\end{equation}
for all $A\in \mathbb{N}^{<\infty}$ and scalars $(a_n)_{n\in A}$.

\begin{proof}
Let $x, A, B, \varepsilon, \delta$ be chosen as in Definition \ref{SLC2defi}. Suppose our basis is quasi-greedy with constant $C_1$, suppression quasi-greedy with constant $C_2$, and $\lambda$-democratic of type $2$ with constant $C_3$. Using \eqref{e9} and $\lambda$-democracy of type $2$, we have
$$\|1_{\varepsilon A}\|\ \leqslant\ 2C_1\|1_A\|\ \leqslant\ 2C_1C_3\|1_B\|\ \leqslant\ 4C_1^2C_3\|1_{\delta B}\|\ \leqslant\ 4C_1^3C_3\|x+1_{\delta B}\|,\mbox{ and }$$
$$\|x\|\ \leqslant\ C_2\|x+1_{\delta B}\|.$$
Hence,
$$\|x+1_{\varepsilon A}\|\ \leqslant\ (4C_1^3C_3+C_2)\|x+1_{\delta B}\|,$$
which proves that our basis is $\lambda$-SLC2.
\end{proof}

\begin{proof}[Proof of Theorem \ref{m2'}]
The equivalence between i) and ii) is due to Lemmas \ref{l1} and \ref{l2}. That iii) implies ii) follows from Proposition \ref{p1}, while to see the implication from ii) to iii), we simply set $x = 0$ in the definition of $\lambda$-SLC2. 
\end{proof}

\subsection{A $\lambda$-AG2 basis that is not AG - Proof of Theorem \ref{m10}}
We borrow the basis from \cite[Section 4]{C3}, which is the unit vector basis of a space whose norm involves permutation of two different weights assigned to different subsequences of $\mathbb{N}$. In particular, let $D = \{2^n: n\in\mathbb{N}\}$, $u_n = 1/\sqrt{n}$, and $v_n = 1/n$ for all $n\geqslant 1$. Let $X_w$ be the completion of $c_{00}$ with respect to the following norm:
$$\|x:=(x_n)_{n=1}^\infty\|\ =\ \sup_{\pi,\pi'}\left(\sum_{n\in D}u_{\pi(n)}|x_n| + \sum_{n\notin D}v_{\pi'(n)}|x_n|\right),$$
where the sup is taken over all bijections $\pi: D\rightarrow \mathbb{N}$ and $\pi':\mathbb{N}\backslash D\rightarrow\mathbb{N}$. Let $\mathcal{B}_w$ be the unit vector basis of the space $X_w$. Clearly, $\mathcal{B}_w$ is unconditional with constant $1$ and normalized. Thanks to Theorems \ref{thmDKKT} and \ref{m2'}, it suffices to verify that given any $\lambda > 1$, $\mathcal{B}_w$ is $\lambda$-democratic of type $2$ but not democratic, both of which can be confirmed using the same proofs as in \cite[Section 4]{C3}. Indeed, the proof of \cite[Claim 4.2]{C3} does not utilize the fact that $B<A$, so the same proof can be used to show that $\mathcal{B}_w$ is $\lambda$-democratic of type $2$. To see that $\mathcal{B}_w$ is not democratic, we simply, for $N\in\mathbb{N}$, let $A_N = \{2, 2^2, \ldots, 2^N\}$ and $B_N = \{3^{N+1}, 3^{N+2}, \ldots, 3^{2N}\}$. By the definition of $\|\cdot\|$, we know that $\|1_{A_N}\| = \sum_{n=1}^{N}1/\sqrt{n} \sim \sqrt{N}$ and $\|1_{B_N}\| = \sum_{n=1}^N 1/n \sim \ln N$. Hence, $\lim_{N\rightarrow\infty}\|1_{A_N}\|/\|1_{B_N}\| = \infty$, and so $\mathcal{B}_w$ is not democratic.

\subsection{The equivalence of $\lambda_1$-AG2 and $\lambda_2$-AG2}\label{equiv}
\begin{proof}[Proof of Proposition \ref{m2}]
By Theorem \ref{m2'}, it suffices to show that a $\lambda_2$-democratic of type $2$ and quasi-greedy basis is $\lambda_1$-democratic of type $2$. Assume that our basis is $\lambda_2$-democratic of type $2$ with constant $C_1$ and quasi-greedy with constant $C_2$. Take two nonempty sets $A, B\in \mathbb{N}^{<\infty}$ with $(\lambda_1-1)s(A) + |A|\leqslant |B|$ and $B$ surrounding $A$. Partition the interval $[\min A, \max A]$ into consecutive intervals
\begin{align*}
I_1 &\ = \ \left[\min A, \min A + \left\lceil\frac{\lambda_1-1}{\lambda_2-1}s(A)\right\rceil - 1\right],\\
I_2 &\ =\ \left[\min A +  \left\lceil\frac{\lambda_1-1}{\lambda_2-1}s(A)\right\rceil, \min A + 2\left\lceil\frac{\lambda_1-1}{\lambda_2-1}s(A)\right\rceil - 1\right],\\
\vdots\\
I_{k_0} &\ =\ \left[\min A +  (k_0-1)\left\lceil\frac{\lambda_1-1}{\lambda_2-1}s(A)\right\rceil, \min A + k_0\left\lceil\frac{\lambda_1-1}{\lambda_2-1}s(A)\right\rceil - 1\right],\\
I_{k_0+1} &\ =\ [\min A, \max A]\backslash \left(\cup_{j=1}^{k_0} I_j\right),
\end{align*}
where $k_0$ is the largest nonnegative integer such that 
$$\min A + k_0\left\lceil\frac{\lambda_1-1}{\lambda_2-1}s(A)\right\rceil - 1 \ \leqslant\ \max A.$$
It is easy to see that
\begin{equation}\label{e23}k_0\ \leqslant\ \left\lceil \frac{\lambda_2-1}{\lambda_1 - 1}\right\rceil.\end{equation}
For $1\leqslant j\leqslant k_0+1$, define $A_j := A\cap I_j$. It follows that 
$$s(A_j)\ \leqslant\ \left\lceil \frac{\lambda_1-1}{\lambda_2-1}s(A)\right\rceil.$$
Define $B' = B\cup D$, where $D > A\cup B$ and $|D| = \lceil \lambda_2-1\rceil$. 
For each $j$, $B'$ surrounds $A_j$, and 
\begin{align*}
(\lambda_2-1)s(A_j) + |A_j| &\ \leqslant\ (\lambda_2-1)\left\lceil \frac{\lambda_1-1}{\lambda_2-1}s(A)\right\rceil + |A|\\
&\ \leqslant\ (\lambda_1-1)s(A) + \lambda_2 - 1 + |A|\\
&\ \leqslant\ |B| + \lambda_2 - 1\ \leqslant \ |B'|.
\end{align*}
Hence, by $\lambda_2$-democracy of type $2$, we obtain
$\|1_{A_j}\| \leqslant C_1\|1_{B'}\|$. By the triangular inequality,
\begin{equation}\label{e20}\|1_A\|\ \leqslant\ \sum_{j=1}^{k_0+1} \|1_{A_j}\|\ \leqslant\ C_1(k_0+1)\|1_{B'}\|.\end{equation}
Furthermore, if we let $C_3:= \inf_n \|e_n\|$ and $C_4:= \sup_n \|e_n\|$, then for an $m\in B$,
\begin{align}\label{e21}\|1_{B'}\|&\ \leqslant\ \|1_B\| + \|1_{D}\|\ \leqslant\ \|1_B\| + C_4\lceil \lambda_2-1\rceil,\mbox{ and } \\
\label{e22}\|1_B\|&\ \geqslant\ \frac{1}{C_2}\|e_{m}\| \ \geqslant\ \frac{C_3}{C_2}.
\end{align}
It follows from \eqref{e23}, \eqref{e20}, \eqref{e21}, and \eqref{e22} that 
$$\|1_A\|\ \leqslant\ C_1\left(\left\lceil \frac{\lambda_2-1}{\lambda_1-1}\right\rceil+1\right)\left(1+C_2\lceil \lambda_2-1\rceil \frac{C_4}{C_3}\right)\|1_B\|.$$
We, therefore, conclude that the basis is $\lambda_1$-democratic of type $2$. 
\end{proof}
\section{Isometric theory for $\lambda$-almost greedy of type $2$ bases}\label{isometry}
We would like to characterize bases that satisfy \eqref{e5} using $\lambda$-SLC2 with constant $1$. It turns out that
when $\lambda \leqslant 2$, we have the equivalence
$$\lambda\mbox{-AG2 with constant 1}\Longleftrightarrow\ \lambda\mbox{-SLC2 with constant 1}.$$
However, for every $\lambda > 2$, the equivalence fails; in particular, we have only the backward implication. This section proves these claims. We start with an easy observation that holds for all $\lambda \geqslant 1$. 

\subsection{Every $\lambda\geqslant 1$}
\begin{prop}\label{implysuppquasi}
A basis that is $\lambda$-SLC2 with constant $1$ must be suppression quasi-greedy with constant $1$. 
\end{prop}

\begin{proof}
Setting $A = \emptyset$, $C = 1$, and $B = j$ in Definition \ref{SLC2defi}, we have a basis that satisfies
\begin{equation}\label{e31}\|x\|\ \leqslant\ \|x+e_j\|,\end{equation}
for every $x\in X$ with $\|x\|_\infty\leqslant 1$ and for every $j\notin \supp(x)$. By induction, \eqref{e31} implies that our basis is suppression quasi-greedy with constant $1$. 
\end{proof}

From Proposition \ref{implysuppquasi} and Lemma \ref{l2}, we deduce
\begin{cor}\label{coback} We have the implication
$$\lambda\mbox{-AG2 with constant 1}\Longleftarrow\ \lambda\mbox{-SLC2 with constant 1}.$$
\end{cor}
\subsection{Below the threshold $\lambda \leqslant 2$}
We aim at proving the following
\begin{lem}\label{l20}
For $\lambda\leqslant 2$, 
\begin{equation}\label{e32}\lambda\mbox{-AG2 with constant 1}\Longleftrightarrow\ \lambda\mbox{-SLC2 with constant 1}.\end{equation}
\end{lem}

The backward implication of \eqref{e32} holds for all $\lambda$, as stated in Corollary \ref{coback}. To prove the forward implication, we recall the proof of Lemma \ref{l1}. There we show that if a basis is $\lambda$-AG2 with constant $C_1$, then it has to be $\lambda$-SLC2 with constant $C_1C_2$, where $C_2$ is the suppression quasi-greedy constant. Therefore, to prove the forward implication of \eqref{e32}, it suffices to prove that for $\lambda\leqslant 2$,
$$\lambda\mbox{-AG2 with constant 1}\ \Longrightarrow \ \mbox{suppression quasi-greedy with constant }1.$$

\begin{proof}[Proof of Lemma \ref{l20}]
Assume that our basis $\mathcal{B}$ is $\lambda$-AG2 with constant $1$. 
We have
\begin{equation}\label{e33}\|x-G_2(x)\|\ \leqslant\ \|x-e_j^*(x)e_j\|, \forall x\in X,  j\in \mathbb{N},  G_2(x).\end{equation}
Indeed, if $\lambda = 1$, we let $m = 2$ in \eqref{e5}; if $\lambda\in (1, 2]$, we let $m = 1$. 
Pick arbitrary $y\in X$ with $\|y\|_\infty\leqslant 1$ and $k\notin \supp(y)$. Set $x:= y+e_k+e_\ell$, where $\ell$ is any natural number such that $\ell\notin \{k\}\cup \supp(y)$. By \eqref{e33}, 
$$\|x-(e_k+e_\ell)\|\ \leqslant\ \|x-e_\ell\|;$$
hence,
$$\|y\|\ \leqslant\ \|y+e_k\|,$$
which is true for all $y\in X$ with $\|y\|_\infty\leqslant 1$ and $k\notin \supp(y)$. We proceed by induction to conclude that $\mathcal{B}$ is suppression quasi-greedy with constant $1$. 
\end{proof}

\subsection{Above the threshold $\lambda > 2$}
Fix $\lambda > 2$. We shall construct a basis that is not suppression quasi-greedy with constant $1$, but the basis is $\lambda$-AG2 with constant $1$. By Proposition \ref{implysuppquasi}, the basis is not $\lambda$-SLC2 with constant $1$. 

Let $X$ be the completion of $c_{00}$ with respect to the norm
$$\|x=(x_i)_{i=1}^\infty\| \ =\  \left|\frac{x_1}{\lambda}+x_2\right|\vee \left|x_1+\frac{x_2}{\lambda}\right|\vee \frac{1}{\lambda}\sum_{i\geqslant 1} |x_i|.$$
Let $\mathcal{B}$ be the unit vector basis.

\begin{claim}
The basis $\mathcal{B}$ is not suppression quasi-greedy with constant $1$.
\end{claim}

\begin{proof}
Choose $s\in \left(\max\left\{\frac{1}{\lambda-1}, \frac{\lambda}{\lambda+1}\right\}, 1\right)$, which is possible due to $\lambda > 2$. Set $x = (-s, 1, 0, \ldots)$ and $y = (-s, 0, 0, \ldots)$. We have
$$\|x\| \ =\ \left(-\frac{s}{\lambda}+1\right)\vee \left|-s+\frac{1}{\lambda}\right|\vee\frac{1}{\lambda}(1+s)\mbox{ and }\|y\|\ =\ s.$$
By how $s$ is chosen, $\|x\| < \|y\|$, so $\mathcal{B}$ is not suppression quasi-greedy with constant $1$. 
\end{proof}

\begin{claim}
   The basis $\mathcal{B}$ is $\lambda$-AG2 with constant $1$.  
\end{claim}
\begin{proof}
Let $x = (x_i)_{i=1}^\infty\in X, m\in \mathbb{N}$ and choose $A\in \mathcal{G}(x, \lceil\lambda m\rceil)$ and $I\in \mathcal{I}$ with $|I|\leqslant m$. We need to show that 
$$\|x-P_A(x)\|\ \leqslant\ \|x-P_I(x)\|.$$
We proceed by case analysis. 

Case 1: If $\{1,2\}\subset A$, then 
$$\|x-P_A(x)\|\ =\ \frac{1}{\lambda}\sum_{i\notin A}|x_i|\ \leqslant\ \frac{1}{\lambda}\sum_{i\notin I}|x_i|\ \leqslant\ \|x-P_I(x)\|.$$

Case 2: If $|\{1,2\}\cap A|=1$, we can assume without loss of generality that $1\in A$. It follows that
$$\|x-P_A(x)\|\ =\ |x_2| \vee \frac{1}{\lambda}\sum_{i\notin A}|x_i|.$$
Since $A$ is a greedy set of order $\lceil \lambda m\rceil$ and $2\notin A$, there exist at least $\lceil \lambda m \rceil+1$ coefficients $x_i$'s that are at least $|x_2|$. Hence,
$$\|x-P_I(x)\|\ \geqslant \ \frac{1}{\lambda}\sum_{i\notin I}|x_i| \ \geqslant\ \frac{\lceil \lambda m \rceil+1-m}{\lambda}|x_2|\ \geqslant\ |x_2|,$$
which implies that $\|x-P_A(x)\|\leqslant \|x-P_I(x)\|$.

Case 3: If $\{1,2\}\cap A = \emptyset$, then
$$\|x-P_A(x)\|\ =\ \left|\frac{x_1}{\lambda}+x_2\right|\vee \left|x_1 + \frac{x_2}{\lambda}\right| \vee \frac{1}{\lambda}\sum_{i\notin A}|x_i|.$$
Since $A$ is a greedy set of order $\lceil \lambda m\rceil$ and $\{1, 2\}\cap A = \emptyset$, there exist at least $\lceil \lambda m \rceil+1$ coefficients $x_i$'s that are at least $t:=\max\{|x_1|, |x_2|\}$. 

If $x_1$ and $x_2$ have opposite signs, 
\begin{align*}\left|\frac{x_1}{\lambda} + x_2\right|, \left|x_1 + \frac{x_2}{\lambda}\right|\ \leqslant\ t\ \leqslant\ \frac{\lceil \lambda m \rceil+1-m}{\lambda}t\ \leqslant\ \frac{1}{\lambda}\sum_{i\notin I}|x_i|\ \leqslant\ \|x-P_I(x)\|. 
\end{align*}

If $x_1$ and $x_2$ have the same sign, we assume that $x_1, x_2 > 0$. Then it suffices to show the following inequality 
\begin{equation}\label{e30}\max\left\{\frac{x_1}{\lambda} + x_2, x_1+\frac{x_2}{\lambda}\right\}\ \leqslant\ \|x-P_I(x)\|.\end{equation}
Inequality \eqref{e30} is immediate if $\{1,2\}\cap I = \emptyset$. Assume that $\{1,2\}\cap I \neq \emptyset$.

Case 3a: $\{1, 2\}\subset I$. Then $m\geqslant 2$. Recall that there exist at least $\lceil \lambda m \rceil$ coefficients $x_i$'s with $i\geqslant 3$ and of modulus at least $t$. Hence,
$$\frac{1}{\lambda}\sum_{i\notin I}|x_i|\ \geqslant\ \frac{1}{\lambda}(\lceil \lambda m \rceil - (m - 2))t\ \geqslant\ 2t\ \geqslant \ \max\left\{\frac{x_1}{\lambda} + x_2, x_1+\frac{x_2}{\lambda}\right\}.$$

Case 3b: $\{1,2\}\cap I = \{1\}$. In this case, $I = \{1\}$ because $I$ is an interval. It follows that 
$$\frac{1}{\lambda}\sum_{i\notin I}|x_i|\ =\ \frac{1}{\lambda}\sum_{i\geqslant 2}|x_i|\ \geqslant\ \frac{1}{\lambda}(\lceil \lambda m \rceil t + x_2)\ \geqslant\ t+\frac{x_2}{\lambda}\ \geqslant \ \max\left\{\frac{x_1}{\lambda} + x_2, x_1+\frac{x_2}{\lambda}\right\}.$$

Case 3c: $\{1,2\}\cap I = \{2\}$. The same reasoning as in Case 3a gives
$$\frac{1}{\lambda}\sum_{i\notin I}|x_i|\ \geqslant\ \frac{x_1}{\lambda} + \frac{1}{\lambda}(\lceil \lambda m \rceil - (m - 1))t\ \geqslant\ \frac{x_1}{\lambda}+\left(m-\frac{m}{\lambda}+\frac{1}{\lambda}\right)t.$$
Hence,
$$\frac{1}{\lambda}\sum_{i\notin I}|x_i|\ \geqslant\ \frac{x_1}{\lambda}+\left(m-\frac{m}{\lambda}+\frac{1}{\lambda}\right)x_2\ \geqslant\ \frac{x_1}{\lambda} + x_2,$$
It remains to show that
$$\frac{x_1}{\lambda}+\left(m-\frac{m}{\lambda}+\frac{1}{\lambda}\right)t\ \geqslant\ x_1+\frac{x_2}{\lambda}.$$
\begin{itemize}
\item If $x_1\geqslant x_2$, we have
\begin{align*}
\frac{x_1}{\lambda}+\left(m-\frac{m}{\lambda}+\frac{1}{\lambda}\right)t&\ =\ \frac{x_1}{\lambda}+\left(m-\frac{m}{\lambda}+\frac{1}{\lambda}-1\right)x_1+x_1\\
&\ \geqslant\ \left(m-\frac{m}{\lambda}+\frac{2}{\lambda}-1\right)x_2+x_1\ \geqslant\ x_1 + \frac{x_2}{\lambda}.
\end{align*}
\item If $x_1 < x_2$, we have
\begin{align*}
\frac{x_1}{\lambda}+\left(m-\frac{m}{\lambda}+\frac{1}{\lambda}\right)t&\ =\ \frac{x_1}{\lambda}+\left(m-\frac{m}{\lambda}+\frac{1}{\lambda}\right)x_2\\
&\ >\ \left(\frac{1}{\lambda}+m-\frac{m}{\lambda}\right)x_1 + \frac{x_2}{\lambda}\ \geqslant\ x_1 + \frac{x_2}{\lambda}. 
\end{align*}
\end{itemize}
This completes our proof. 
\end{proof}

\section{On the $\lambda$-PG2 and $\lambda$-RPG2 properties}\label{(R)PG}
As briefly mentioned in Section \ref{mainprelim}, enlarging greedy sums in the original definition of RPG bases by Dilworth and Khurana does not weaken the property (\cite[Theorem 5.8]{C1}). This motivates \cite{C3} to find a reformulation of the RPG property such that an analog of Theorem \ref{pm3} exists for the reformulation. The original definition of RPG bases \cite{DK} involves the natural ordering within greedy sums. For our purposes, we use the reformulation in \cite{C3} instead. 

\begin{thm}\cite[Theorem 1.1]{C3}\label{RPGchar} A basis is RPG if and only if there exists $C\geqslant 1$ such that
\begin{equation}\label{e3}\|x-P_\Lambda(x)\|\ \leqslant\ C\inf\{\|x-P_I(x)\|\,:\, I\in \mathcal{I}, |I|\leqslant m, \mbox{ and if }I\neq \emptyset, \Lambda \leqslant \max I\},\end{equation}
for all $x\in X$, $m\in \mathbb{N}$, and $\Lambda\in \mathcal{G}(x, m)$.
\end{thm}

\begin{defi}\cite[Definition 1.3]{C3} \normalfont
A basis is said to be $\lambda$-reverse partially greedy of type $2$ ($\lambda$-RPG2) if it satisfies \eqref{e3}
for all $x\in X$, $m\in \mathbb{N}$, and $\Lambda\in \mathcal{G}(x, \lceil \lambda m\rceil)$. Observe that $\mathcal{G}(x, m)$ is replaced by $\mathcal{G}(x, \lceil\lambda m\rceil)$.
\end{defi}

An analog of Theorem \ref{RPGchar} holds for PG bases.
\begin{thm}\cite[Theorem 1.2]{C3}\label{PGchar} A basis is PG if and only if there exists $C\geqslant 1$ such that
\begin{equation}\label{e26}\|x-P_\Lambda(x)\|\ \leqslant\ C\inf\{\|x-P_I(x)\|\,:\, I\in \mathcal{I}, |I|\leqslant m, \mbox{ and if }I\neq \emptyset, \Lambda \geqslant \min I\},\end{equation}
for all $x\in X$, $m\in \mathbb{N}$, and $\Lambda\in \mathcal{G}(x, m)$.
\end{thm}

\begin{defi}\normalfont\label{defLambdaPG2}
A basis is said to be $\lambda$-partially greedy of type $2$ ($\lambda$-PG2) if it satisfies \eqref{e26} for all $x\in X$, $m\in \mathbb{N}$, and $\Lambda\in \mathcal{G}(x, \lceil \lambda m\rceil)$. 
\end{defi}

The same argument used in Subsection \ref{equiv} gives
\begin{prop}\label{nn}
For $1<\lambda_1 < \lambda_2$, the following holds
\begin{enumerate}
\item[i)] A basis is $\lambda_1$-PG2 if and only if it is $\lambda_2$-PG2.
\item[ii)] A basis is $\lambda_1$-RPG2 if and only if it is $\lambda_2$-RPG2. 
\end{enumerate}
\end{prop}

We are ready to prove \eqref{e25}.
\begin{proof}[Proof of \eqref{e25}]
Let $\lambda_2 > \lambda_1 > 1$ and $\lambda_1', \lambda_2', \lambda_1'', \lambda_2'' > 1$. 

Proposition \ref{m2} states that $\lambda_1''$-AG2 is the same as $\lambda_2''$-AG2. That $\lambda_2''$-AG2 implies $\lambda_1'$-PG2 follows immediately from Definitions \ref{defiAG2} and \ref{defLambdaPG2} and Proposition \ref{nn}. Let us discuss a basis that is $\lambda_1'$-PG2 but not $\lambda_2''$-AG2. Such a basis recently appeared in the proof of \cite[Proposition 6.10]{BDKOW}. Let $X_{S}$ be the completion of $c_{00}$ with respect to the norm
$$\|x = (x_i)_{i=1}^\infty\|\ =\ \sup \left\{\sum_{i\in F}|x_i|\,:\, \sqrt{\min F}\geqslant |F|\right\}.$$
Consider the unit vector basis $\mathcal{B}_S$ of $X_S$. It is easy to verify that $\mathcal{B}_S$ is PG. By Theorem \ref{PGchar}, we know that $\mathcal{B}_S$ is $1$-PG2 and thus, $\lambda_1'$-PG2. However, $\mathcal{B}_S$ is not $\lambda_2''$-AG2 because it is not $2$-democratic of type $2$ (recall Corollary \ref{ocor}). To see this, pick $A = \{N^2, N^2, \ldots, N^2+N-1\}$ and $B = \{1, \ldots, 2N\}$. Then $|B|\geqslant s(A) + |A|$, and $B$ surrounds $A$ for large $N$, but 
$$\frac{\|1_B\|}{\|1_A\|}\ \sim\ \frac{\sqrt{2N}}{N}\ \rightarrow\ 0\mbox{ as }N\rightarrow\infty.$$

That $\lambda_1'$-PG2 is the same as $\lambda_2'$-PG2 is due to Proposition \ref{nn}. According to Theorem \ref{pm1}, $\lambda_1$-PG implies $\lambda_2$-PG; however, the basis $\mathcal{B}_{\lambda_1, \lambda_2}$ in Section \ref{differentiatePG} shows that $\lambda_2$-PG does not necessarily imply $\lambda_1$-PG. 

It remains to show that $\lambda'_2$-PG2 implies $\lambda_1$-PG, but the converse is not necessarily true. By Proposition \ref{nn}, a $\lambda'_2$-PG2 basis is also $\lambda_1$-PG2. It follows from Definitions \ref{defLambdaPG} and \ref{defLambdaPG2} that a $\lambda_1$-PG2 basis is $\lambda_1$-PG. Hence, $\lambda'_2$-PG2 implies $\lambda_1$-PG. To see that the converse is not necessarily true, let us consider the basis $\mathcal{B}_{\lambda_1-\varepsilon, \lambda_1}$ in Section \ref{differentiatePG}, where $\varepsilon \in (0, \lambda_1 - 1)$. Then $\mathcal{B}_{\lambda_1-\varepsilon, \lambda_1}$ is $\lambda_1$-PG but not $(\lambda_1-\varepsilon)$-PG and thus, is not $(\lambda_1-\varepsilon)$-PG2. By Proposition \ref{nn}, $\mathcal{B}_{\lambda_1-\varepsilon, \lambda_1}$ is not $\lambda'_1$-PG2. 
\end{proof}
\section{Conclusion}
We have shown that all of the following greedy-type properties: greedy, AG, PG, and RPG are weakened when we enlarge greedy sums in either their original definitions or their equivalent reformulations that involve intervals (see Theorem \ref{m10} and \cite[Theorem 1.4]{C3}). Out of all these properties, only the PG property gives us a continuum of PG-type properties as the enlarging factor $\lambda$ varies in $[1, \infty)$; for the greedy, AG, and RPG properties, varying $\lambda\in (1,\infty)$ makes no difference (see \cite[Theorem 3.3]{DKKT}, Theorem \ref{m1}, and Propositions \ref{m2} and \ref{nn}).

\ \\
\end{document}